\documentclass[12pt]{article}
\usepackage{amsmath,amssymb,amsthm, amsfonts}
\usepackage{hyperref}
\usepackage{graphicx}
\usepackage{array, tabulary}
\usepackage{url}
\usepackage[mathlines]{lineno}
\usepackage{dsfont} 
\usepackage{color}
\usepackage{subcaption}
\usepackage{enumitem}
\definecolor{red}{rgb}{1,0,0}

\definecolor{blue}{rgb}{0,0,1}

\definecolor{green}{rgb}{0,.6,0}

\usepackage{float}

\usepackage{tikz}

\newtheorem{thm}{Theorem}[section]
\newtheorem{cor}[thm]{Corollary}
\newtheorem{lem}[thm]{Lemma}
\newtheorem{prop}[thm]{Proposition}

\newtheorem{obs}[thm]{Observation}

\newtheorem{quest}[thm]{Question}

\theoremstyle{definition}

\theoremstyle{definition}
\newtheorem{defn}[thm]{Definition}

\theoremstyle{definition}

\newcommand{\bit}{\begin{itemize}}
\newcommand{\eit}{\end{itemize}}
\newcommand{\ben}{\begin{enumerate}}
\newcommand{\een}{\end{enumerate}}
\newcommand{\beq}{\begin{equation}}
\newcommand{\eeq}{\end{equation}}
\newcommand{\bea}{\begin{eqnarray}} 
\newcommand{\eea}{\end{eqnarray}}
\newcommand{\bpf}{\begin{proof}}
\newcommand{\epf}{\end{proof}\ms}
\newcommand{\bmt}{\begin{bmatrix}}
\newcommand{\emt}{\end{bmatrix}}
\newcommand{\ms}{\medskip}

\newcommand{\lf}{\left\lfloor}
\newcommand{\rf}{\right\rfloor}

\newcommand{\noi}{\noindent}

\newcommand{\beqs}{\begin{equation*}} 
\newcommand{\eeqs}{\end{equation*}}
\newcommand{\beas}{\begin{eqnarray*}}
\newcommand{\eeas}{\end{eqnarray*}}

\newcommand{\thc}{\operatorname{th}_c}
\newcommand{\thd}{\operatorname{th}_d}
\newcommand{\capt}{\operatorname{capt}}
\newcommand{\dmg}{\operatorname{dmg}}
\newcommand{\rad}{\operatorname{rad}}
\newcommand{\srg}{\operatorname{SRG}}

\title{The damage throttling number of a graph}
\author{Joshua Carlson \thanks{Dept.~of Mathematics and Statistics, Williams College, Williamstown, MA, USA (\{jc31, re1, jmp10,  ps15\}@williams.edu)}
\and 
Robin Eagleton \footnotemark[1]
\and 
Jesse Geneson \thanks{Dept.~of Mathematics and Statistics, San Jos\'e State University, San Jos\'e, CA, USA (jesse.geneson@sjsu.edu) }
\and
John Petrucci \footnotemark[1]
\and 
Carolyn Reinhart \thanks{Dept.~of Mathematics, Iowa State University, Ames, IA, USA (reinh196@iastate.edu) Research is supported by NSF grant DMS-1839918} 
\and
Preetul Sen \footnotemark[1]
}

\date{\today}

\begin{document}
\maketitle

\begin{abstract} 
The \emph{cop throttling number} of a graph, introduced in 2018 by Breen et al., optimizes the balance between the number of cops used and the number of rounds required to catch the robber in a game of Cops and Robbers. In 2019, Cox and Sanaei studied a variant of Cops and Robbers in which the robber tries to occupy (or damage) as many vertices as possible and the cop tries to minimize this damage. In their paper, they study the minimum number of vertices damaged by the robber over all games played on a given graph $G$, called the \emph{damage number} of $G$. We introduce the natural parameter called the \emph{damage throttling number} of a graph, denoted $\thd(G)$, which optimizes the balance between the number of cops used and the number of vertices damaged in the graph. To this end, we formalize the definition of \emph{$k$-damage number}, which extends the damage number to games played with $k$ cops. We show that damage throttling and cop throttling share many properties, yet they exhibit interesting differences. We prove that the damage throttling number is tightly bounded above by one less than the cop throttling number. Infinite families of examples and non-examples of tightness in this bound are given. For most families of connected graphs $G$ of order $n$ that we consider, we find that $\thd(G) = O(\sqrt{n}).$ However, we find an infinite family of connected graphs $G$ of order $n$ for which $\thd(G) = \Omega(n^{2/3})$.  
\end{abstract}

\noi {\bf Keywords} Cops and Robbers, Capture time, Damage number, Throttling

\noi{\bf AMS subject classification} 05C57, 05C15, 05C50

\begin{section}{Introduction}\label{intro}
Cops and Robbers is a two-player pursuit-evasion game played on simple graphs which was introduced in \cite{AF84, NW83, Q78}. In this game, a team of $k$ cops attempt to capture a single robber on a given graph. In round $0$, each cop and the robber choose a vertex to occupy, starting with the cops. In all subsequent rounds, each cop either stays in their location or moves along an edge of the graph, after which the robber has the same choice.\footnote{This is equivalent to playing on a reflexive graph (every vertex has a loop) and requiring the robber and all cops to move during each round. Both interpretations are used in the literature.} If during any round, a cop occupies the same vertex as the robber, the cops win and the robber is caught (or captured). Alternatively, if the robber has a strategy to avoid capture forever, the robber wins the game.

Many graph parameters naturally arise from the game of Cops and Robbers. First introduced in \cite{AF84}, the minimum number $c(G)$ of cops required on a graph $G$ to guarantee capture of the robber is called the \emph{cop number} of the graph. A graph $G$ with $c(G) = 1$ is called \emph{cop-win}. Currently, one of the biggest open problems in Cops and Robbers is \emph{Meyniel's Conjecture} \cite{frankl1987cops}, which states that $c(G) = O(\sqrt{n})$.

Other parameters consider the amount of time taken to play a game of Cops and Robbers. For a graph $G$ and integer $k \geq 1$, the \emph{$k$-capture time} of a graph $G$, denoted $\capt_k(G)$, is the minimum number of rounds required for $k$ cops to capture a robber over all games played on $G$ where the robber avoids capture for as long as possible. The $k$-capture time of a graph was first studied for $k = c(G)$ in \cite{capttime} and for other $k$ values in \cite{BPPR17}. Note that if $k < c(G)$, then $\capt_k(G)$ is defined to be infinity. In \cite{CRthrottle}, the \emph{cop throttling number}, $\thc(G)$, of a graph $G$ was introduced in order to study the optimal balance between the number of cops used and their capture time. Specifically, for any graph $G$ on $n$ vertices, $\thc(G) = \underset{1 \leq k \leq n}{\min}\{k + \capt_k(G)\}$. The cop throttling number is an upper bound for the cop number of a graph and it was asked in \cite{CRthrottle} whether $\thc(G) = O(\sqrt{n})$. This question was answered negatively in \cite{CopThrot2}.

Recently in \cite{CS19}, Cox and Sanaei introduced an interesting parameter $\dmg(G)$, called \emph{the damage number} of $G$. A vertex is considered \emph{damaged} if the robber ever occupies that vertex in a round in which capture does not occur. For a graph $G$, $\dmg(G)$ is the minimum number of vertices damaged over all games played on $G$ with a single cop where the robber places and plays to maximize damage. Note that in this variant of the classic game, the cops try to minimize damage and do not necessarily capture the robber. Cox and Sanaei also mention that damage can be studied with multiple cops; the following definition formalizes this idea.
\begin{defn}\label{def:kdmg}
Suppose $G$ is a graph on $n$ vertices and $k$ is an integer with $1 \leq k \leq n$. The \emph{$k$-damage number} of $G$, denoted $\dmg_k(G)$, is the minimum number of vertices damaged over all games of Cops and Robbers played on $G$ with $k$ cops where the robber places and plays in order to maximize damage. 
\end{defn}

In contrast to $k$-capture time, the $k$-damage number of a graph is still interesting when $k < c(G)$ since $k$ cops can always seek to minimize damage regardless of whether capture is possible. Furthermore, the notion of $k$-damage number leads to the investigation of the optimal balance between the number of cops used and the number of vertices damaged. This idea is captured in the next definition.

\begin{defn}\label{defn:dmgthrot}
Suppose $G$ is a graph on $n$ vertices. The \emph{damage throttling number} of $G$ is defined as $\thd(G) = \underset{1 \leq k \leq n}{\min}\{k + \dmg_k(G)\}$.
\end{defn}

In this paper, we study the damage throttling number of a graph and how it compares to the cop throttling number. In Section \ref{sec:dmgCopThrot}, we explore various similarities between the two throttling numbers. First, we prove that for any connected graph $G$, $c(G) \leq \thd(G)$ and we show that for any graph $G$, $\thd(G) \leq \thc(G) - 1$. In Section \ref{sec:largedamagethrot}, we find a family of connected graphs with large damage throttling number. We find in Section \ref{subsec:boundtight} several infinite families of graphs $G$ for which $\thd(G)=\thc(G) - 1$. Furthermore, we show this equality holds for all connected graphs on $6$ or fewer vertices in Section \ref{subsec:fewvertices}. In Section \ref{sec:diff}, we explore the differences between damage throttling and cop throttling by finding infinite families of graphs $G$ for which $\thd(G) < \thc(G) - 1$. In particular, we examine a family of graphs with highest possible capture time in Section \ref{subsec:hn}. Finally, we study $k$-damage numbers and how they compare to the cop number in Section \ref{sec:damageNumbers}. 
Throughout this paper, we follow most of the graph theoretic and Cops and Robbers notation found in \cite{Diestel} and \cite{CRbook} respectively. 
\end{section}

\begin{section}{Similarities with cop throttling}\label{sec:dmgCopThrot}


In this section, we examine some ways in which the damage throttling number of a graph $G$ is similar to the cop throttling number of $G$. First, recall that the domination number of a graph $G$, $\gamma(G)$, is the smallest cardinality of a subset $S \subseteq V(G)$ such that $V(G)$ is the closed neighborhood of $S$.  It was observed in \cite{CRthrottle} that for any graph $G$, $\thc(G) \leq \gamma(G) + 1$, and we now make a similar observation for the damage throttling number.

\begin{obs}
    For any graph $G$, $\thd(G)\leq\gamma(G)$.
\end{obs}

Next, we investigate the relationship between the damage throttling number and the cop number of a graph. Although for all graphs $G$, $c(G)\leq \thc(G)$ is trivial, this is not the case with damage throttling. For example, the graph $\overline{K_n}$ consisting of $n$ isolated vertices satisfies $\thd(\overline{K_n}) = 2 < n = c(\overline{K_n}).$ However, for connected graphs, we can prove the analogous result for damage throttling.

\begin{prop}\label{prop:c(G)leqthd(G)} If $G$ is a connected graph, then $c(G) \leq  \thd(G).$
\end{prop}




\bpf 
Let $G$ be a connected graph on $n$ vertices. We exhibit a strategy for $k + \dmg_k(G)$ cops to guarantee capture of the robber on $G$. Throughout this strategy, $k$ cops (called \emph{active} cops) play optimally to prevent damage on $G$.  The active cops can ensure that the robber can damage at most $\dmg_k(G)$ vertices. We refer to the remaining $\dmg_k(G)$ cops as \emph{undercover} cops. Whenever a new vertex is damaged by the robber, an undercover cop moves to occupy that vertex. Since the robber can damage at most $\dmg_k(G)$ vertices, one of the cops (active or undercover) is guaranteed to be able to capture the robber. Thus, for all $k$, $k+\dmg_k(G)$ cops can capture the robber in finite rounds. In particular, for an integer $\ell$ that realizes $\thd(G)$, $c(G) \leq \ell+\dmg_{\ell}(G) = \thd(G)$.
\epf
We now establish a relationship between $\thd(G)$ and $\thc(G)$ using the following lemma.

\begin{lem}\label{dmgCaptIneqLemma}
If $G$ is a graph on $n$ vertices and $1 \leq k \leq n$ is an integer, then \[\dmg_k(G) \leq \capt_k(G)-1.\]
\end{lem}

\begin{proof}
If $k<c(G)$, then $\capt_k(G) = \infty$ and $\dmg_k(G) < n < \capt_k(G) - 1$. Now, suppose $c(G) \leq k \leq n$. The robber can damage at most one new vertex each round before being caught. If $k$ cops play optimally to catch the robber on $G$, the robber is caught in round $\capt_k(G)$. Therefore, the robber can damage at most $\capt_k(G) - 1$ vertices.
\end{proof}

\begin{prop}\label{prop:INEQthcthd}
For any graph $G$ with $|V(G)| \geq 2$, $\thd(G) \leq \thc(G)-1$.
\end{prop}
\begin{proof}
Suppose $G$ is a graph of order $n \geq 2$. Choose an integer $c(G) \leq \ell \leq n$ that realizes $\thc(G)$; in other words, choose $\ell$ such that $\ell+\capt_{\ell}(G)=\underset{1\leq k \leq n}{\min}\{ k+\capt_k(G) \}$. Then,

\begin{align}
\label{dmgeq1} \thd(G) &= \underset{1 \leq k \leq n}{\min}\{k + \dmg_{k}(G)\} \nonumber \\[.5 em]
&\leq \ell + \dmg_{\ell}(G) \\[.5 em] 
\label{dmgeq2}&\leq \ell + \capt_{\ell}(G) - 1 \quad \quad \text{(by Lemma \ref{dmgCaptIneqLemma})}\\[.5 em] &= \thc(G) - 1. \nonumber \qedhere
\end{align}
\end{proof}

The following corollary characterizes when the bound in Proposition \ref{prop:INEQthcthd} is tight. 


\begin{cor}\label{cor:characterization}
The equality $\thd(G)=\thc(G)-1$ holds if and only if there exists an $\ell \geq c(G)$ such that $\dmg_{\ell}(G) =\capt_{\ell}(G) - 1$ and both $\thc(G)$ and $\thd(G)$ can be achieved with $\ell$ cops. 
\end{cor}

\begin{proof}
Consider each inequality in the proof of Proposition \ref{prop:INEQthcthd}. By definition, an integer $\ell$ realizes $\thc(G)$ if and only if $\ell + \capt_{\ell}(G) = \thc(G)$. Furthermore, $\ell$ realizes $\thd(G)$ if and only if inequality (\ref{dmgeq1}) is tight. Thus, requiring that inequality (\ref{dmgeq2}) is also tight completes the characterization of graphs $G$ with $\thd(G) = \thc(G) - 1$.
\end{proof}

\subsection{Graphs with large damage throttling number}\label{sec:largedamagethrot}
In this subsection, we show that there exist connected graphs with high cop number and high damage number. Since $c(G) \leq \thc(G)$ for any graph $G$, if it were true that $\thc(G) = O(\sqrt{n})$ for all connected graphs, Meyniel's conjecture would follow. However, graphs have been found that have cop throttling number asymptotically larger than $\sqrt{n}$.

\begin{cor}\cite[Corollary 2.3]{CopThrot2}\label{cor:HighCopThrot}
 There exist connected graphs of order $n$ with cop throttling number $\Omega(n^{2/3})$.
\end{cor}

We have shown in Propositions \ref{prop:c(G)leqthd(G)} and \ref{prop:INEQthcthd} that $c(G) \leq \thd(G) \leq \thc(G) - 1$ for all connected graphs $G$, so it is natural to ask whether $\thd(G) = O(\sqrt{n})$ as this would also imply Meyniel's conjecture. Note that the argument used to prove Corollary \ref{cor:HighCopThrot} relies on the fact that capture of the robber is required to achieve a finite cop throttling number. However, this argument is not sufficient for damage throttling since capture is not required to achieve a finite damage throttling number. To that end, we prove the next lemma which is used in the main result of this subsection. It may also be of independent interest, since it strengthens the result that there exist connected graphs of order $n$ with cop number $\Omega(\sqrt{n})$.

\begin{lem}\label{dmg_lem}
For any $n$ sufficiently large and constant $c < 1$, there exists a connected graph of order $n$ where the robber can safely damage $\Omega(n)$ vertices and evade capture forever against at most $c \sqrt{n}$ cops.
\end{lem}

\begin{proof}
We modify the Moore graph $H$ of degree $d = \lf \sqrt{n-1} \rf$ and diameter $2$, which is a regular graph of girth $5$ and order $d^2+1$. This graph is known to have cop number $d$ (in general $c(G) \geq \delta(G)$ for any $G$ of girth at least $5$). Let the graph $G$ of order $n$ be obtained from $H$ by picking a single vertex $v$ of $H$ and adding a path $P$ of length $n-1-d^2$ that is connected to $v$ at an endpoint. 

If $k \leq c \sqrt{n}$, then there exists a constant $r > 0$ such that $d-k > r \sqrt{n}$ for $n$ sufficiently large since $c < 1$. The robber uses the following strategy to safely damage $\Omega(n)$ vertices of $H$ in $G$ and evade capture forever against $k$ cops.

After the cops make their initial placements, the robber views any cops on the path $P$ as being on $v$ instead, ignoring the path. Since $d-k > r \sqrt{n}$, the number of vertices occupied by or adjacent to a cop at any time is at most $d-1+(d-1)d < d^2+1$, so the robber can choose an initial position not occupied by any cop or adjacent to any cop. 

Call a vertex \emph{guarded} if it is occupied by or adjacent to a cop. Since the Moore graph has diameter $2$, every vertex in $H$ is within distance $2$ of the robber's current vertex. Since the Moore graph avoids $C_3$ and $C_4$, each cop guards at most one neighbor of the robber. During every \emph{odd} round, the robber goes to its unguarded neighbor with the fewest number of damaged neighbors. The robber always has at least $r \sqrt{n}$ unguarded neighbors. 

During every \emph{even} round, the robber goes to any unguarded undamaged neighbor. If in any even round $2i$ the robber has no unguarded undamaged neighbor, then every unguarded neighbor of the robber's vertex in the previous odd round $2i-1$ had at least $r \sqrt{n}$ damaged neighbors. The closed neighborhoods of the unguarded neighbors of the robber's vertex $u$ have no intersection besides $u$ since $H$ avoids $C_3$ and $C_4$, so there must be at least $(r \sqrt{n}) (r \sqrt{n}-1) = \Omega(n)$ damaged vertices if the robber is ever forced to go to a damaged vertex in an even round. Thus the robber can safely damage $\Omega(n)$ vertices in $G$.
\end{proof}

The next theorem has almost the same proof as the corresponding theorem in \cite{CopThrot2}, with the main difference being that in this construction we use the graphs from Lemma \ref{dmg_lem} instead of arbitrary graphs with $\Omega(\sqrt{n})$ cop number, and the robber uses a specific evasion strategy that safely damages many vertices instead of an arbitrary evasion strategy.

\begin{thm}
For $n$ sufficiently large, there exist connected graphs $X$ of order $n$ with $\thd(X) = \Omega(n^{2/3})$.
\end{thm}

\begin{proof}
We construct a connected graph $X$ with $\thd(X) = \Omega(n^{2/3})$ by starting with a spider of order $n$ with $\lf n^{1/3}\rf$ legs of length approximately $n^{2/3}$. Then, we replace approximately half of each leg farthest from the center vertex by a copy of one of the graphs $G$ from Lemma \ref{dmg_lem} with the same order as the replaced vertices. Each copy is connected by a single vertex to the end of the leg that remains. If the number of cops does not already give the bound $\thd(X) = \Omega(n^{2/3})$, then the robber will start on the copy of $G$ on the leg of the spider with the fewest cops. It will take at least $\Omega(n^{2/3})$ rounds for any cops on other legs to reach the robber's leg, so the robber can use the strategy in Lemma \ref{dmg_lem} to safely visit a new undamaged vertex in every even round until it has damaged at least $\Omega(n^{2/3})$ vertices. 
\end{proof}

\begin{subsection}{Graph families such that $\thd(G)=\thc(G)-1$}\label{subsec:boundtight}

Although Corollary \ref{cor:characterization} provides a complete characterization of graphs $G$ that satisfy $\thd(G)=\thc(G)-1$, the given conditions are not easy to verify. Therefore, further study of this equality is useful. In order to find several families of graphs that achieve this equality, we now turn our attention to the $k$-radius of a graph and use the following result.

\begin{prop}\cite{BPPR17}\label{prop:radkCopThrot}
If $G$ is a connected graph on $n$ vertices and $1 \leq k \leq n$ is an integer, then $\capt_k(G)\geq \rad_k(G).$
\end{prop}

The proof of Proposition \ref{prop:radkCopThrot} uses a stationary robber, but such a strategy is not optimal for damage. We now prove the analogous result using a different robber strategy.
\begin{prop}\label{prop:radk}
If $G$ is a connected graph on $n$ vertices and $1 \leq k \leq n$ is an integer, then $\dmg_k(G)\geq \rad_k(G) - 1.$
\end{prop}
\bpf
First, note that if $\rad_k(G) \leq 1$, then $\dmg_k(G) \geq 0$ is trivially satisfied. Next, suppose $\rad_k(G) \geq 2$ and consider an arbitrary initial placement of $k$ cops on a subset $S\subseteq V(G)$ of vertices. Choose a vertex $x\in V(G)$ such that $d(S,x)$ is maximum. Choose $u\in S$ such that $d(u,x) = d(S, x) $. Let $P$ be a shortest path in $G$ from $u$ to $x$. Place the robber on the vertex $y$ in $V(P)$ such that $d(y,u)=2$. If the robber moves towards $x$ along the path $P$ in each round, then $|V(P)|-2$ vertices are damaged. Since $|V(P)|-1\geq \rad_k(G)$, this means $\dmg_k(G)\geq \rad_k(G)-1$.
\epf

As noted in \cite[Remark 3.2]{CRthrottle}, Proposition \ref{prop:radkCopThrot} yields $\thc(G) \geq \underset{1 \leq k \leq n}{\min}\{k+\rad_k(G)\}$ as an immediate corollary. Proposition \ref{prop:radk} leads us to the following analogous result for $\thd(G)$. 

\begin{cor}\label{cor:thdANDradk}

For any graph $G$, $\thd(G) \geq \underset{1 \leq k \leq n}{\min}\{k + \rad_k(G)\} - 1.$
\end{cor}

We have established in Proposition \ref{prop:INEQthcthd} that for any nontrivial graph $G$, $\thd(G) \leq \thc(G) -1$.  While we are also interested in graphs where $\thd(G) < \thc(G) -1$ (see Section \ref{sec:diff}), we now turn our attention to instances when this bound is an equality. Using our previous results about $k$-radius, we show the desired equality holds for several classes of graphs.
 
\begin{prop}\label{prop:minradk}
If $\thc(G)=\underset{1 \leq k \leq n}{\min}\{k + \rad_k(G)\}$, then $\thd(G)=\thc(G) - 1$.
\end{prop}

\bpf
Suppose $\thc(G) = \underset{1 \leq k \leq n}{\min}\{k + \rad_k(G)\}.$  By Corollary \ref{cor:thdANDradk}, \[\thd(G) \geq \underset{1 \leq k \leq n}{\min}\{k + \rad_k(G)\} -1 = \thc(G)-1.\]  Since $\thd(G) \leq \thc(G)-1$ by Proposition \ref{prop:INEQthcthd}, it follows that $\thd(G) =\thc(G)-1.$
\epf

Next, we apply known results about graphs $G$ for which $\thc(G)=\underset{1 \leq k \leq n}{\min}\{k + \rad_k(G)\}$ in order to show that $\thd(G)=\thc(G)-1$ for these graphs. Recall that a chordal graph is a graph in which every induced cycle is a $C_3$.

\begin{prop}\cite{CRthrottle}
For any tree or cycle $G$ on $n$ vertices, $\thc(G)=\underset{1 \leq k \leq n}{\min}\{k + \rad_k(G)\}$.
\end{prop}

\begin{cor}\label{cor:treesAndCycles}
For any tree or cycle $G$, $\thd(G)=\thc(G)-1$.
\end{cor}

\begin{prop}\cite{CopThrot2}\label{chordRadk}
For any connected chordal graph $G$ and integer $1 \leq k \leq |V(G)|$, $\capt_k(G)=\rad_k(G)$.
\end{prop}

\begin{cor}\label{cor:chordalThrottling}
For any connected chordal graph $G$, $\thd(G)=\thc(G)-1$.
\end{cor}

It is worth noting that the converse of Proposition \ref{prop:minradk} does not hold; that is, there exist graphs such that $\thd(G)=\thc(G)-1$ and $\thc(G)>\underset{1 \leq k \leq n}{\min}\{k + \rad_k(G)\}$. The Petersen graph $P$ provides such an example, which we will examine next. First, recall that a graph $G$ is $\srg(n, k, \lambda, \mu)$ if $|V(G)| = n$, $G$ is $k$-regular, every pair of adjacent vertices in $G$ has $\lambda$ common neighbors, and every pair of non-adjacent vertices in $G$ has $\mu$ common neighbors. The well-known fact that $P$ is $\srg(10, 3, 0, 1)$ is particularly useful for determining $\thd(P)$.

\begin{thm}
For the Petersen graph $P$, $\underset{1 \leq k \leq n}{\min}\{k + \rad_k(P)\}=3$, $\thc(P)=4$, and $\thd(P) = 3$.
\end{thm}

\bpf
In order to find $\underset{1 \leq k \leq n}{\min}\{k + \rad_k(P)\}$, note that $\rad(P) = 2$ and $\gamma(P) = 3$. This gives the following values of $\rad_k(P)$:

\[
\rad_k(P)=
\begin{cases}
    2 &\text{if } k=1,2; \\
    1 &\text{if } 3\leq k\leq 9; \\
    0 &\text{if } k=10.
\end{cases}
\]
So we see that $\underset{1 \leq k \leq n}{\min}\{k + \rad_k(P)\} = 3$.

Next, note that $c(P) = 3$ \cite{AF84} and  since $c(P) = \gamma(P)$, we have the following capture times: 
\[
\capt_k(P)=
\begin{cases}
    \infty &\text{if } k=1,2; \\
    1 &\text{if } 3\leq k\leq 9; \\
    0 &\text{if } k=10.
\end{cases}
\]
These capture times imply that $\thc(P) = 4$.

Finally, we calculate $\thd(P)$ by considering all possible damage numbers. We know that $\dmg_1(P)=5$ \cite{CS19} and since $\gamma(P) = 3$, $\dmg_k(P)=0$ for all integers $3\leq k\leq 10$. We now prove that $\dmg_2(P)=2$ by showing that the robber can always damage two vertices, and that the cops can always prevent the robber from damaging a third vertex. Since $P$ is $\srg(10,3,0,1)$, each pair of adjacent vertices has no common neighbors and each pair of non-adjacent vertices has exactly one common neighbor. Thus, while one cop dominates four vertices, two adjacent cops dominate six vertices and two non-adjacent cops dominate seven.

To show that the robber can always damage two vertices, we note that regardless of cop placement, the robber will always be adjacent to a vertex not dominated by either of the cops. Otherwise, the cops could move such that one of them dominates two of the robber's neighbors, which contradicts $P$ being $\srg(10,3,0,1)$. So, in round $1$, the robber moves to an adjacent non-dominated and undamaged vertex, thus damaging their starting vertex. In round $2$, the robber damages the vertex they occupy and moves to an adjacent non-dominated vertex. Note that the vertex the robber moves to may be its original starting vertex. Thus, $\dmg_2(P)\geq 2$.

To show that the cops can prevent a third damaged vertex, place the cops on non-adjacent vertices so that they dominate seven vertices. Playing according to the robber strategy above, at the start of round $2$, the robber is on a non-dominated vertex. Further, the robber is adjacent to three vertices, namely $u$, the now-damaged starting vertex; $v$, an undamaged but dominated vertex; and $w$, which is undamaged and may or may not be dominated.

If $w$ is dominated, the cops can stay still in round $2$, which forces the robber back to $u$. This damages a second vertex, but the robber is now on a previously damaged vertex. If $w$ is not dominated, then a cop that dominates $v$ stays still, while the other cop moves to dominate $w$. This move is possible since every vertex not adjacent to $w$ has a common neighbor with $w$. Thus, the only non-dominated vertex adjacent to the robber is $u$, and so the robber must move to $u$. By repeating this strategy, the cops restrict the robber's movement to these two damaged vertices. Thus, $\dmg_2(P)\leq 2$.

Therefore, $\dmg_2(P)=2$ and we have the following:
\[
\dmg_k(P)=
\begin{cases}
    5 &\text{if } k=1; \\
    2 &\text{if } k=2; \\
    0 &\text{if } 3\leq k\leq 10.
\end{cases}
\]
These damage numbers imply that $\thd(P)=3$.
\epf

\end{subsection}

\begin{subsection}{Graphs on few vertices}\label{subsec:fewvertices}

In this subsection, we explore the gap between damage and cop throttling in graphs with few vertices. As we have seen previously, when $\gamma(G)$ cops play optimally on a graph $G$, the robber is captured in the first round and no vertices are damaged. As such, we can expect that for graphs with small enough domination numbers, we will be able to restrict the gap between damage and cop throttling to just $1$.

\begin{lem}\label{lem:domNum2}
If $G$ is a nontrivial connected graph with $\gamma(G)\leq 2$, then $\thd(G)=\thc(G)-1$.
\end{lem}
\begin{proof}
Suppose $G$ is a graph on $n$ vertices. If $\gamma(G)=1$, then $\capt_k(G)=1$ and $\dmg_k(G)=0$ for all $k<n$; this gives $\thd(G)= 1 = \thc(G)-1$. If $\gamma(G)=2$, then $\capt_1(G)\geq2$ and $\capt_k(G)=1$ for all $2\leq k<n$. Furthermore, we know $\dmg_1(G)\geq1$ and $\dmg_k(G)=0$ for all $2\leq k<n$. Together, this implies that $\thd(G)=2 = \thc(G)-1.$
\end{proof}

Most nontrivial graphs with order at most $6$ have a domination number of $2$. Thus, Lemma \ref{lem:domNum2} and additional consideration for those graphs with $\gamma(G)=3$ allow us to classify all nontrivial connected graphs on at most $6$ vertices as exhibiting a difference of $1$ between damage and cop throttling.

\begin{prop}\label{prop:order6Gap}
If $G$ is a connected graph of order $2\leq n\leq6$, then $\thd(G)=\thc(G)-1$.
\end{prop}

\begin{proof}
It is well-known that $\gamma(G) \leq \frac{n}{2}$ for all connected graphs $G$ with $|V(G)| \geq 2$. If $|V(G)| \leq 5$, then $\gamma(G) \leq 2$ and by Lemma \ref{lem:domNum2}, $\thd(G)=\thc(G)-1$.

Suppose now that $|V(G)|=6$; then, $\gamma(G) \leq \frac{6}{2} = 3$. If $\gamma(G)\leq 2$, then $\thd(G)=\thc(G)-1$ by Lemma \ref{lem:domNum2}. Using the \textit{Sage} code in \cite{code_domnum3}, we find that the only two graphs of order $6$ with $\gamma(G)=3$ are those illustrated in Figure \ref{fig:order6dom3}. 
\begin{figure}[H] \begin{center}
\scalebox{.7}{\includegraphics{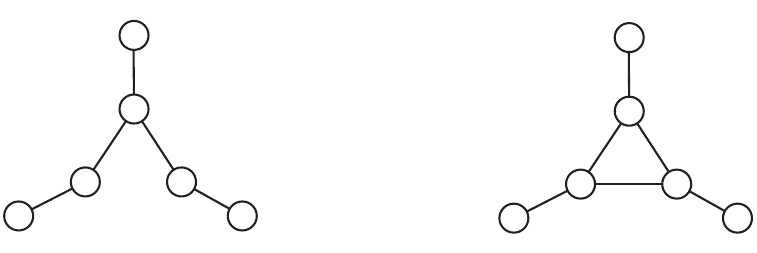}}\\
\caption{The two order-$6$ graphs $G$ with $\gamma(G) = 3$.}\label{fig:order6dom3} 
\end{center}
\end{figure}

Observe that the graph on the left in Figure \ref{fig:order6dom3} is a tree and the graph on the right is chordal. Therefore, by Corollary \ref{cor:treesAndCycles} and Corollary \ref{cor:chordalThrottling} respectively, these two graphs also have $\thd(G)=\thc(G)-1$ as desired. \qedhere

\end{proof}

Thus, any nontrivial connected graph of at most $6$ vertices will exhibit the lowest possible gap of $1$ between its damage and cop throttling numbers. This motivates us to consider what the minimum order is of graphs with a difference of more than $1$ between $\thd(G)$ and $\thc(G)$. More generally, we are interested in finding graphs that exhibit this larger gap of at least $2$.

\end{subsection}
\end{section}


\begin{section}{Differences between damage and cop throttling}\label{sec:diff}

In this section, we turn to examining various graphs for which the gap between damage throttling and cop throttling is at least two. While these are harder to find, infinite families of graphs that realize this gap do exist; in particular, we demonstrate two infinite families in which $\thc(G)$ and $\thd(G)$ remain constant and one in which they grow without bound. A graph exhibiting a gap of three is also presented.

First, we continue our discussion of graphs on few vertices by showing that the smallest graphs exhibiting a gap of two consist of only $7$ vertices.

\begin{prop}
The connected graphs of smallest order such that $\thd(G)\leq\thc(G)-2$ have $7$ vertices; there are thirteen of them in total.

\end{prop}

\begin{proof}
By Proposition \ref{prop:order6Gap}, $\thd(G)=\thc(G)-1$ for any connected graphs of order $6$, so it suffices to consider graphs on order $7$.

For any $7$-vertex graph $G$, $\gamma(G)\leq\frac72$, so $\gamma(G)\in\{1,2,3 \}$. By Lemma \ref{lem:domNum2}, if $G$ has domination number $1$ or $2$, then $\thd(G)=\thc(G)-1$. Thus, we need only consider graphs $G$ of order $7$ with $\gamma(G)=3$. To find these, we will algorithmically check every $7$-vertex graph to see if it can be dominated by three vertices. Using the \textit{Sage} code in \cite{code_domnum3}, we find forty-two such graphs that have $\gamma(G)=3$. Of these, we will show that the twenty-nine graphs displayed in Figure \ref{fig:29GapOfOne} have a gap of only one, but the thirteen graphs in Figure \ref{fig:13GapOfTwo} have the desired gap of two.

\begin{figure}[ht]
    \centering
    \begin{tabular}{|c|m{12.5cm}|}
        \hline
        trees & \includegraphics[scale=.3]{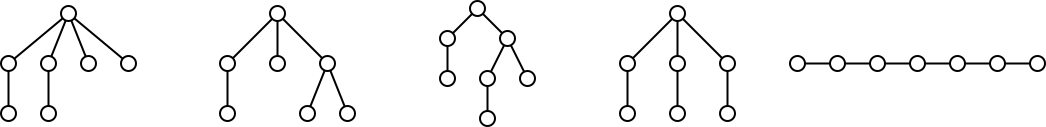} \\
        \hline
        chordal graphs & \includegraphics[scale=.3]{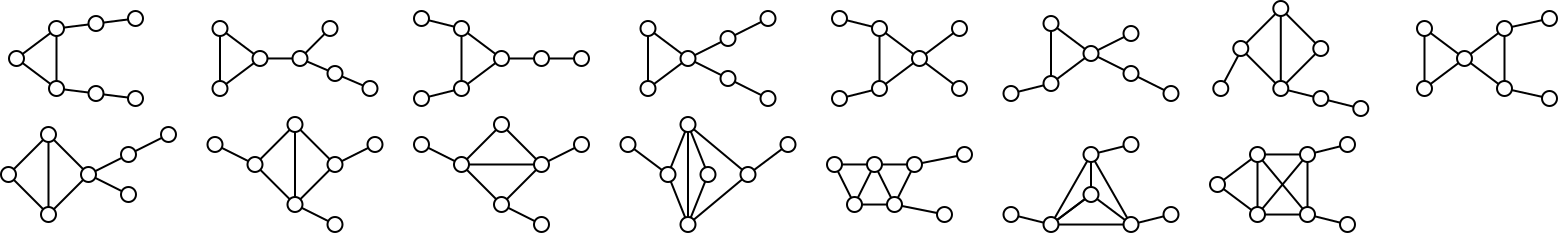}\\
        \hline
        cycle & \includegraphics[scale=.3]{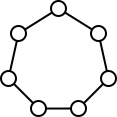}\\
        \hline
        $\dmg_1(G)>1$ & \includegraphics[scale=.3]{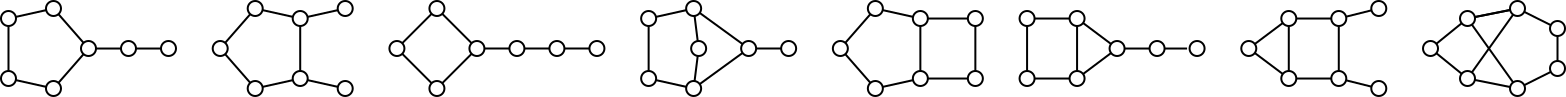}\\
        \hline
    \end{tabular}
    \caption{The twenty-nine order-$7$ graphs $G$ with $\gamma(G) = 3$ but $\thd(G)=\thc(G)-1$.}
    \label{fig:29GapOfOne}
\end{figure}

\begin{figure}[ht]
    \centering
    \includegraphics[scale=.45]{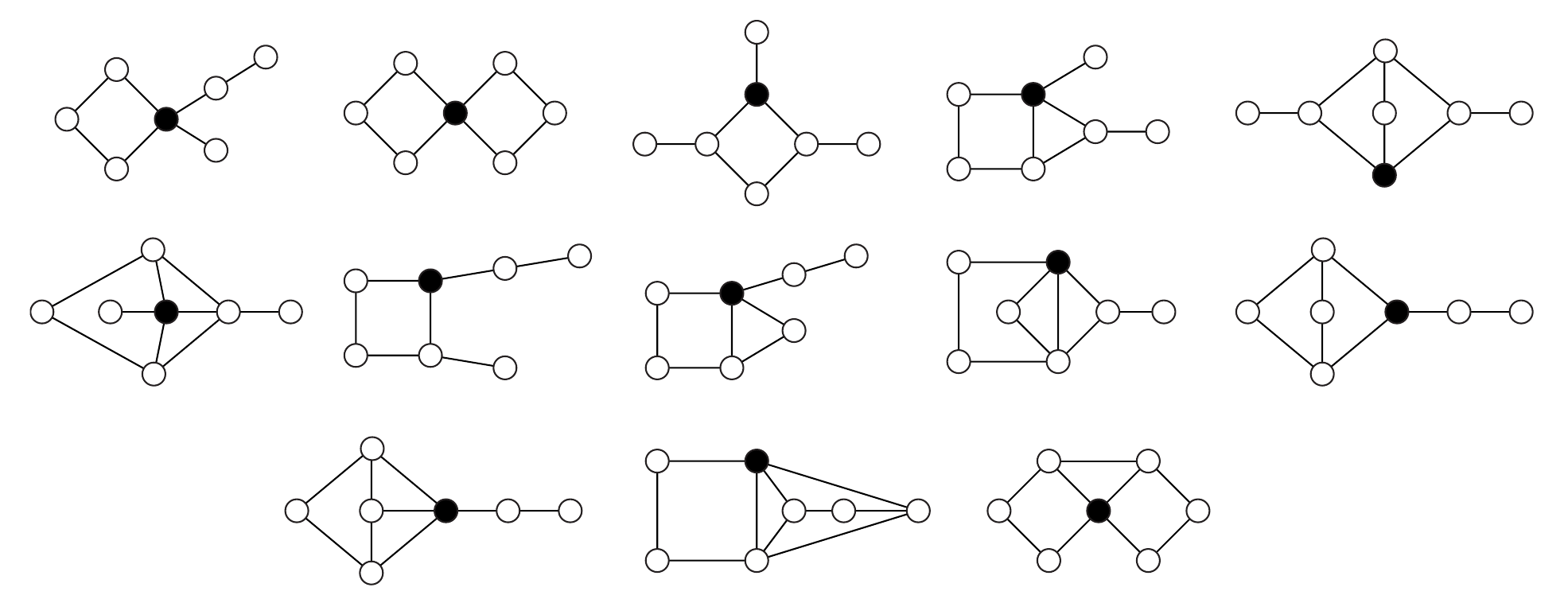}
    \caption{The thirteen order-$7$ graphs $G$ with $\gamma(G) =3$ and $\thd(G)=\thc(G)-2$, with a vertex indicated in black in each graph on which the cop can place to ensure $\dmg_1(G)=1$.}
    \label{fig:13GapOfTwo}
\end{figure}


Observe that of the graphs in Figure \ref{fig:29GapOfOne}, five are trees, fifteen are chordal, and one is a cycle; by Corollaries \ref{cor:treesAndCycles} and \ref{cor:chordalThrottling}, these twenty-one graphs have $\thd(G)=\thc(G)-1$.

For each of the remaining eight graphs in Figure \ref{fig:29GapOfOne} and thirteen graphs in Figure \ref{fig:13GapOfTwo}, we will directly calculate $\thc(G)$ and $\thd(G)$. First, by the Python code in \cite{code_dismantlable}, none of these graphs are dismantlable, and so by \cite{CRbook}, they are not cop-win. Since $\gamma(G)=3$ for each of these graphs, $\capt_2(G)\geq 2$ and $\capt_k(G)=1$ for $3\leq k\leq 6$. Therefore $\thc(G)=4$.

We now calculate $\thd(G)$. Since $\gamma(G)=3$, we have $\dmg_2(G)\geq 1$ and $\dmg_k(G)=0$ for $3\leq k\leq 7$. Observe that if $\dmg_1(G)=1$, then $G$ will have $\thd(G)=2$ and achieve the desired gap; otherwise, if $\dmg_1(G)>1$, then $\thd(G)=3$ and the gap is only one. We thus wish to characterize which of these twenty-one graphs have $\dmg_1(G)=1$.

Note that in each of the thirteen graphs in Figure \ref{fig:13GapOfTwo}, if the cop places on the black vertex, then by passing in each subsequent round, the cop restricts the robber on only damaging one vertex. Thus, for these graphs, $\dmg_1(G)=1$ and so $\thd(G)=\thc(G)-2$, as desired. The eight graphs in the last row of Figure \ref{fig:29GapOfOne} do not contain such a vertex the cop can place on; this implies the robber can move to a new vertex in round $1$ without getting captured in round $2$. As such, $\dmg_1(G)>1$ and $\thd(G)=\thc(G)-1$ for these eight graphs.
\end{proof}

If we consider disconnected graphs, the gap between damage and cop throttling can grow arbitrarily far apart, as exhibited by the following infinite family.

\begin{obs}
For $n \geq 2$,  $\thd(\overline{K_n}) = 2$ and $\thc(\overline{K_n}) - 1 = n - 1$. Thus, as $n$ increases, $\thd(\overline{K_n})$ and $\thc(\overline{K_n}) - 1$ get arbitrarily far apart.
\end{obs}

However, we would like to find examples of infinite families of connected graphs that exhibit the desired gap. While graphs with a dominating vertex $v$ will not realize such a gap, we can carefully reduce the number of vertices that $v$ dominates to restrict the robber's motion and lower the $k$-damage numbers.  This motivates the following definitions which are illustrated in Figure \ref{fig:gearaccordion}.

\begin{defn}
Recall the definition of the wheel graph $W_n$ of order $n$ as the graph obtained by adding a dominating vertex, known as the hub, to the cycle $C_{n-1}$. For each integer $\ell \geq 2$, the \emph{gear graph} of order $2\ell + 1$ is denoted $G_{2 \ell + 1}$ and is obtained from $W_{2 \ell + 1}$ by deleting every other edge incident to the hub. 
\end{defn}
\begin{defn}
Further, recall the definition of the fan graph $F_n$ as the graph obtained by adding a dominating vertex to a the path $P_{n-1}$. We can now analogously define the \emph{accordion graph} for each integer $\ell \geq 2$, denoted $A_{2 \ell}$, as the graph obtained from $F_{2\ell}$ by deleting every other edge incident to the dominating vertex.
\end{defn}
\begin{figure}[H] \begin{center}
\scalebox{.65}{\includegraphics{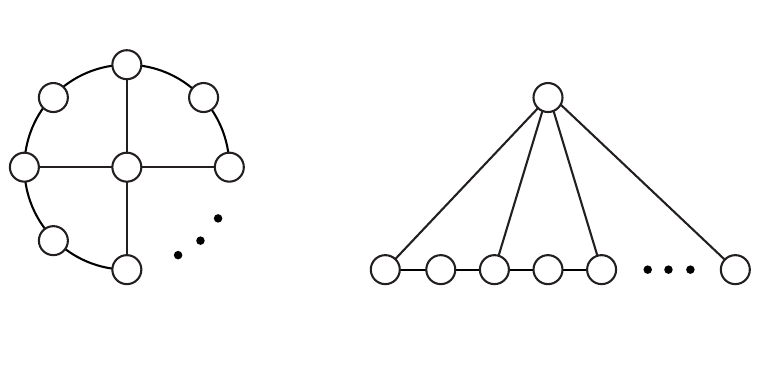}}\\
\caption{The gear and accordion graphs are shown on the left and right respectively.}\label{fig:gearaccordion} 
\end{center}
\end{figure}

Next, we show that the gear and accordion graphs are infinite families of connected graphs for which $\thd(G) = \thc(G) - 2$.

\begin{thm} For each integer $\ell\geq 4$, if $G \in \{G_{2\ell + 1}, A_{2 \ell}\}$, then $\thd(G)=\thc(G)-2$. Thus,
there exist infinitely many connected graphs $G$ such that $\thd(G) < \thc(G) - 1$.
\end{thm}

\bpf
Consider $G_{2\ell+1}$ for any $\ell\geq 4$. By placing and remaining on the hub vertex, one cop can restrict the robber to damaging only a single vertex. Since $G_{2\ell+1}$ does not have a dominating vertex, the robber can always damage at least one vertex, and so we have $\dmg_1(G_{2\ell+1})=1$. This implies $\thd(G_{2\ell+1})\leq2$. Since $\thd(G_{2\ell+1})=1$ only when there is a dominating vertex, $\thd(G_{2\ell+1})=2$.

Observe  that $G_{2\ell+1}$ is not cop-win, so consider $k=2$. Place both cops on the hub vertex. Assuming the robber places optimally on a vertex non-adjacent to the hub, the cops can move to either side of the robber in round $1$ and capture them in round $2$, giving $\capt_2(G_{2\ell+1})\leq 2$. Further, $\capt_2(G_{2\ell+1})>1$, since $\gamma(G_{2\ell+1})>2$ for $\ell\geq 4$. Thus, $\capt_2(G_{2\ell + 1})= 2$ and $\thc(G_{2\ell+1})\leq 4$. Since $\capt_3(G_{2\ell+1})>0$, we see that $\thc(G_{2\ell+1})>3$ and conclude that $\thc(G_{2\ell+1})=4$.

Using the same argument, we can show for the accordion graph $A_{2\ell}$ that $\thd(A_{2\ell})=2$ and $\thc(A_{2\ell})=4$ for all $\ell\geq 4$. Since $G_{2\ell+1}$ and $A_{2\ell}$ are infinite families, there are infinitely many connected graphs such that $\thd(G)<\thc(G)-1$, as desired.
\epf

Notice that for each $G \in \{G_{2\ell+1}, A_{2\ell}\}$, $c(G)=\thd(G)$. The next proposition shows that for graphs $G$ with this property, we can more easily determine whether $\thd(G)<\thc(G)-1$.


\begin{prop}\label{prop:copGammaThrot}
If $c(G)=\thd(G)$, then one of the following is true:
\begin{enumerate}
    \item $\gamma(G)=c(G)$
    \item $\thd(G)<\thc(G)-1$
\end{enumerate}
\end{prop}
\bpf
Let $c(G)=\thd(G)$ and assume $\thd(G)\not<\thc(G)-1$. If $G$ is trivial, $c(G)=\thd(G)=\gamma(G)$. If $G$ is non-trivial, then by Proposition \ref{prop:INEQthcthd}, $\thd(G)=c(G)=\thc(G)-1$ and so either $\capt_{c(G)}(G)=1$ or $\capt_{c(G)+1}(G)=0$. 

If $\capt_{c(G)}(G)=1$, then by definition, some arrangement of $c(G)$ cops dominate the graph and $\gamma(G)\leq c(G)$. However, any dominating set of vertices in a graph form an initial cop placement which catches the robber, so clearly $c(G)\leq \gamma(G)$. Therefore, $\gamma(G)=c(G)$. 

If $\capt_k(G)=0$, then $k=n$. Thus, if $\capt_{c(G)+1}(G)=0$, we must have $c(G)=n-1$ and so G has at least one edge. Recall that $c(G)\leq \gamma(G)$ and for graphs with at least one edge, $\gamma(G)\leq n-1$. Therefore $n-1=\gamma(G)$ and this implies the graph has at most one edge. The only graphs with one edge are the disjoint union of $K_2$ with some number of isolated vertices. However, we can easily observe that for this class of graphs, $\gamma(G)=c(G)$.
\epf




Thus far, we have considered several graphs $G$ that satisfy $\thd(G) = \thc(G) - 2$. We now find a graph that exhibits a larger difference between the cop and damage throttling numbers.

\begin{thm}\label{thm:GapOf3Graph}
There exists a connected graph $G$ with $\thd(G) \leq \thc(G) - 3$.
\end{thm}
\begin{proof}
We will show that the graph $G$ in Figure \ref{fig:GapOf3Graph} has $\thc(G)=6$ and $\thd(G)=3$.

\begin{figure}[H] \begin{center}
\scalebox{.7}{\includegraphics{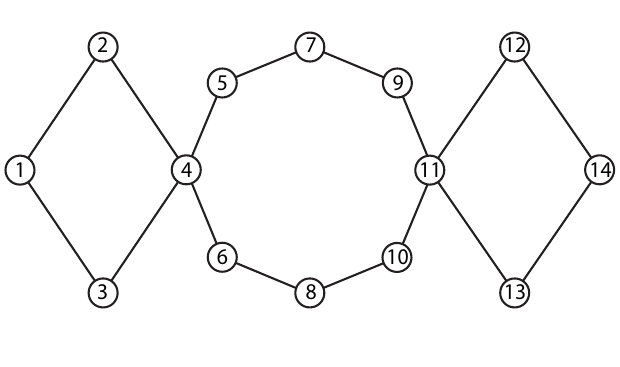}}\\
\caption{A graph $G$ with $\thc(G) - \thd(G) = 3$.}\label{fig:GapOf3Graph} 
\end{center}
\end{figure}

We show first that $\gamma(G)=5$. Note that $\{1,4,9,10,14 \}$ is a five-vertex dominating set. Suppose that a dominating set $S$ of only four vertices exists. In order to dominate vertices $1$, $7$, $8$, and $14$, $S$ must include a vertex from each of $\{1,2,3\}$, $\{5,7,9\}$, $\{6,8,10\}$, and $\{12,13,14\}$, respectively. Since the vertices in these four sets each have degree 2, every vertex in $S$ dominates three vertices and so $S$ can dominate at most twelve vertices. Therefore no such set $S$ of size four exists.


Observe that $G$ is not cop-win; so, we consider $k\geq 2$. By domination, $\capt_k(G)=1$ for $5\leq k\leq 13$ and $\capt_k(G)\geq 2$ for $2\leq k\leq 4$. It remains to show that $\capt_2(G)\geq 4$ and $\capt_3(G)\geq 3$.

Now consider $k=2$. We will show for any cop placement, the robber evades capture for at least four rounds. Given a cop placement, if there exists a vertex that is at least distance $4$ away from both cops, then the robber can stay on that vertex, and will evade capture until at least the fourth round.

So, assume there is no vertex of distance at least $4$ from both cops. We can algorithmically check over all possible cop placements to determine which have this property; for a Python implementation of this, see \cite{code_gapof3}. In total, there are thirty-five such placements, which can be reduced to twelve cases when considering horizontal, vertical, and rotational symmetries of the graph. In each case, the robber's strategy to avoid capture until at least the fourth round is to initially place on
\begin{enumerate}[label=(\alph*)]
    \item vertex 14, if the cop set is \{1,9\}, \{1,11\}, \{2,9\}, \{2,10\}, \{2,11\}, \{4,9\}, \{4,11\}, \{5,9\}, \{5,10\};
    \item vertex 13, if the cop set is \{1,12\}, \{2,12\}; or
    \item vertex 12, if the cop set is \{2,13\}.
\end{enumerate}
The robber should choose an analogous starting vertex for cop sets which are symmetric to these. It is straightforward to verify that the robber cannot be caught in fewer than four rounds given these initial placements. Thus, $\capt_2(G)\geq4$.

For three cops, we proceed similarly to show that no matter where the three cops place, the robber always has a strategy to avoid capture until at least the third round. As before, given an initial cop placement, if there exists a vertex that is at least a distance $3$ from all three cops, then the robber can place on that vertex, and can only be caught in the third round or later.

So assume there is no vertex at least distance $3$ away from all cops. Once again, we will check all possible cop placements to determine which have this property using the Python code in \cite{code_gapof3}. In total, there are sixty-eight such placements, reducible to nineteen cases when accounting for symmetries of the graph. In each case, the robber's strategy to avoid capture until at least the third round is to initially place on
\begin{enumerate}[label=(\alph*)]
    \item vertex 14, if the cop set is \{1,2,11\}, \{1,4,11\}, \{1,5,11\}, \{2,2,11\}, \{2,3,11\}, \{2,4,11\}, \{2,5,11\}, \{2,6,11\}, \{2,7,11\}, \{2,8,11\}, \{4,4,11\}, \{4,5,11\}, \{4,7,11\};
    \item vertex 13, if the cop set is \{1,4,12\}, \{2,4,12\};
    \item vertex 12, if the cop set is \{2,4,13\}; or
    \item vertex 3, if the cop set is \{2,9,11\}, \{2,10,11\}, \{2,11,11\}.
\end{enumerate}
The robber should choose an analogous starting vertex for cop sets which are symmetric to these. It is easy to check that the robber cannot be caught in fewer than three rounds using these initial placements. Thus, $\capt_3(G)\geq3$. We obtain that $\thc(G) = 6$ from the following:
\[
\capt_k(G)
\begin{cases}
    =\infty &\text{if } k=1; \\
    \geq4 &\text{if } k=2; \\
    \geq3 &\text{if } k=3; \\
    \geq2 &\text{if } k=4; \\
    =1 &\text{if } 5\leq k\leq 13; \\
    =0 &\text{if } k=14.
\end{cases}
\]

Now, considering damage, note that $\dmg_k(G)=0$ for $5\leq k\leq 14$. For $k=2$, by placing cops on vertices $4$ and $11$, we restrict the robber to placing and thereafter staying on one of $\{1,7,8, 14\}$. This gives $\dmg_2(G)=1$ and implies that $\dmg_k(G) = 1$ for each $k \in \{3,4\}$. Lastly, if $k=1$, the robber can always place such that they are distance $4$ away from the cop since $\rad(G) = 4$. As such, the robber can damage at least two vertices. In summary:
\[
\dmg_k(G)
\begin{cases}
    \geq2 &\text{if } k=1; \\
    =1 &\text{if } 2\leq k\leq 4; \\
    =0 &\text{if } 5\leq k\leq 14.
\end{cases}
\]
Thus, $\thd(G)=3$, giving a gap of three between the cop and damage throttling numbers.
\end{proof}

\begin{subsection}{The family $H_n$}\label{subsec:hn}
When searching for graphs with $\thd(G)<\thc(G)-1$, it is natural that graphs with high capture time are worth investigating. For each integer $n\geq 7$, let $H_n$ denote the graph illustrated in Figure \ref{fig:HnGraph}. In $\cite{G10}$, it was shown that $\capt(H_n) = n -4$, which can achieved by placing the cop on vertex $3$ and moving the cop along vertices $3,2,4,7,8, \ldots, n$. Furthermore, this is the highest possible $1$-capture time for cop-win graphs on at least seven vertices.  Thus, we dedicate this section to the infinite family $H_n$.

\begin{figure}[H] \begin{center}
\scalebox{1}{\includegraphics{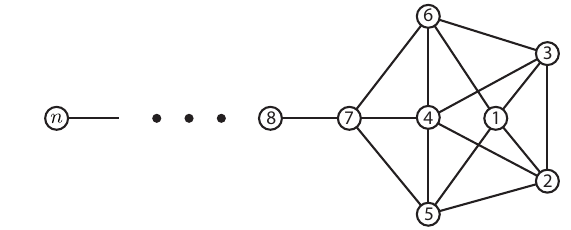}}\\
\caption{The graph $H_n$ for $n \geq 7$.}\label{fig:HnGraph} 
\end{center}
\end{figure}

We refer to the vertices in $\{8,9, \ldots, n\}$ as the \emph{tail} of $H_n$ when $n \geq 8$. As seen with many graphs so far, it is useful to consider the domination number before determining the $k$-capture time and $k$-damage number. This is also the case with $H_n$.

\begin{prop}\label{prop:domHn}
For each integer $n \geq 7$, the domination number of $H_n$ is \[\gamma(H_n) = \left\lceil \frac{n-8}{3} \right\rceil + 2.\]
\end{prop}
\begin{proof}
For $n=7$ and $n=8$, a dominating set is $\{1,7 \}$, so $\gamma(H_n)\leq2$ in these cases. Since there is no dominating vertex in $H_n$, $\gamma(H_n)=2=\lceil \frac{n-8}{3} \rceil + 2$ for $n \in \{7,8\}$. For $n=9$, a dominating set is $\{1,7,9\}$, so $\gamma(H_9)\leq 3$. Suppose we can dominate $H_9$ with only two vertices. To dominate vertex $9$, a dominating set must include vertex $8$ or $9$. However, we cannot dominate vertices $1$ through $6$ with only one vertex, so $\gamma(H_9)>2$. Thus, $\gamma(H_9)=3=\lceil \frac{9-8}{3} \rceil + 2$.

For $n\geq 10$, let $S$ be the set of vertices consisting of $1$, $7$, and every third vertex on the tail of $H_n$ starting at $10$; note that this is $\lceil \frac{n-8}{3} \rceil + 2$ vertices. The set $S$ dominates $H_n$ and so $\gamma(H_n)\leq \lceil \frac{n-8}{3} \rceil + 2$. Suppose that $H_n$ can be dominated by $\lceil \frac{n-8}{3} \rceil + 1$ vertices. In order to dominate the tail, we must include at least every third vertex in a dominating set $S$, which requires $\lceil \frac{n-8}{3} \rceil$ vertices. At most, this will dominate vertices $7, 8, \dots,n$. So, only one more vertex can be added to $S$ to dominate vertices $1, 2, \dots,6$, but no such vertex exists. Therefore, $\gamma(H_n)= \lceil \frac{n-8}{3} \rceil + 2$.
\end{proof}

In our investigation of $H_n$, the next step is computing $\capt_k(H_n)$ and $\thc(H_n)$ as follows.

\begin{lem}\label{lem:kCaptValues}
If $n$ and $k$ are integers with $n \geq 7$ and $1 \leq k \leq n$, then \[\capt_k(H_n) = \begin{cases} \lceil \frac{n-3-k}{2k-1} \rceil & \text{if }1 \leq k < \lceil \frac{n-8}{3} \rceil + 2;\\ 1 & \text{if }\lceil \frac{n-8}{3} \rceil + 2 \leq k < n;\\ 0 & \text{if }k=n. \end{cases}\]
\end{lem}

\begin{proof}
Recall that the domination number of $H_n$ is $\gamma(H_n) = \left\lceil \frac{n-8}{3} \right\rceil + 2$ by Proposition \ref{prop:domHn}, $\capt_k(H_n) = 1$ when $\gamma(H_n) \leq k < n$, and $\capt_ n(H_n) = 0$. Next, consider $2 \leq k < \gamma(H_n)$ and the following strategy for $k$ cops. Place one cop on vertex 3 and distribute the remaining cops on the path $P$ with $V(P) = \{3,2,4,7,8, \ldots , n\}$ to optimize cop throttling on $P$. Let $\mathcal{C}$ be this set of $k$ vertices on which the cops initially place. Note that for any vertex $x \in V(P)$, the distance between $x$ and the nearest cop is at most $\lceil \frac{n-3-k}{2k-1} \rceil$. Let $y$ be the initial position of the cop that is closest to vertex $3$ but not on vertex $3$. If the robber places to the left of $y$, then the robber is guaranteed to be caught in at most $\lceil \frac{n-3-k}{2k-1} \rceil$ rounds. Otherwise, the cop on vertex $3$ can push the robber towards the tail of $H_n$ by moving along the vertices $2,4,7,8,$ and so on. This also guarantees capture of the robber in at most $\lceil \frac{n-3-k}{2k-1} \rceil$ rounds.

For the lower bound, we argue that any initial cop placement other than $\mathcal{C}$ yields a capture time greater than or equal to that of $\mathcal{C}$.  Then, assuming the cops initially place on $\mathcal{C}$, we produce a strategy for the robber that avoids capture for at least $\lceil \frac{n-3-k}{2k-1} \rceil$ rounds. First, observe that from the perspective of the cops, it is always optimal for at most two cops to initially place in the set $\{1, 2, \ldots, 7\}$. This is because $\gamma(H_7) = 2$. Suppose two cops place in the set $\{1, 2, \ldots, 7\}$. Since vertex $7$ is the vertex in $\{1, 2, \ldots 7\}$ that is closest to the tail of $H_n$, it is optimal for the two cops in $H_7$ to place on vertices $7$ and $3$. In this case, if $n \geq 8$, the pigeonhole principle guarantees that there is a vertex $z \in \{8,9, \ldots n\}$ that is at least $\lceil \frac{(n-7)-(k-2)}{2k-3}\rceil \geq \lceil \frac{n-3-k}{2k-1} \rceil$ vertices away from the nearest cop. So the robber can stay on vertex $z$ and avoid capture for at least $\lceil \frac{n-3-k}{2k-1} \rceil$ rounds. Note that if $n = 7$, then $k =2$ and the robber is caught in one round anyway. 

Now, suppose no cops place in the set $\{1, 2, \ldots 7\}$ and thus, $n \geq 9$. Let $C$ be such an initial cop placement in which the rightmost cop is on vertex $x$ and the second rightmost cop is on vertex $y$. Let $C'$ be an initial cop placement obtained from $C$ by moving the rightmost cop to vertex $2$. We show that the capture time with $C’$ is at most the capture time with $C$ assuming optimal play from both players besides the initial position of the cops. We consider the three possible cases for where the robber could have an optimal initial placement. If the robber initially places to the left of $y$, then the modification from $C$ to $C’$ has no effect on the capture time. If the robber initially places to the right of $x$, then the capture time is greater in $C$ than in $C’$, since the robber cannot be captured until a cop reaches vertex $2$ or $3$. If the robber initially places between $x$ and $y$, then a modified robber strategy where the robber initially places to the right of $x$ will achieve a greater capture time for $C$, so the robber would not initially place between $x$ and $y$. Thus, it suffices to consider only initial cop placements where there is exactly one cop on the vertices $\left\{1,2,\dots,7\right\}$.

In the case that exactly one cop places in $\{1, 2, \ldots 7\}$, it is optimal for that cop to start on vertex $3$ so that they can push the robber towards the tail of $H_n$ as quickly as possible. Recall that this is accomplished when that cop moves along the path $P$. In this case, it is optimal for the remaining cops to place so that the capture time on $P$ is minimized. Let $S$ be the set of vertices in a longest subpath of $P$ that does not contain a cop. If $S$ does not contain vertex $2$, then the robber can avoid capture for $\lceil \frac{n-3-k}{2k-1} \rceil$ rounds by placing at the optimal vertex in $S$ and waiting to be captured. If $S$ does contain vertex $2$, then the robber can avoid capture for $\lceil \frac{n-3-k}{2k-1} \rceil$ rounds by starting at vertex $5$ and moving along vertices $7,8,9,$ and so on. In this case, the cop on $3$ must move along $P$ to force the robber into the tail of $H_n$. Note that the capture time is less if the robber starts anywhere else on $S$ because the cop on vertex $3$ can take a shortcut by moving to vertex $4$ in the first round. We can now conclude that $\capt_k(H_n) = \lceil \frac{n-3-k}{2k-1} \rceil$ if $2 \leq k < \gamma(H_n)$. Finally, observe that if $k = 1$, $\capt_k(H_n) = n - 4 = \lceil \frac{n-3-k}{2k-1} \rceil$.\end{proof}

\begin{thm}
For all integers $n \geq 7$, $\thc(H_n) \geq \lceil \sqrt{2n-7}\rceil.$
\end{thm}

\begin{proof}
We wish to minimize $k+\capt_k(H_n)$ over $1\leq k\leq n$. By Lemma \ref{lem:kCaptValues}, $k+\capt_k(H_n)=k+1$ for any $k$ satisfying $\lceil \frac{n-8}{3} \rceil + 2 \leq k < n$; thus, for such values of $k$, we have that $k+\capt_k(H_n)\geq \lceil \frac{n-8}{3} \rceil + 3$. Further, when $k=n$, we have that $k+\capt_k(H_n)=n$.

For $1 \leq k < \lceil \frac{n-8}{3} \rceil + 2 = \gamma(H_n)$, we know $\capt_k(H_n)=\lceil \frac{n-3-k}{2k-1} \rceil$ by Lemma \ref{lem:kCaptValues}. Then,
\[
\min_{1\leq k<\gamma(H_n)}\left\{ k+ \left\lceil \frac{n-3-k}{2k-1} \right\rceil \right\} \geq \min_{1\leq k<\gamma(H_n)}\left\{ k+ \frac{n-3-k}{2k-1} \right\}.
\]
We now determine the minimum value of $f(k):=k+ \frac{n-3-k}{2k-1}$. Note that
\begin{eqnarray*}
    f'(k) 
     &=& \frac{(2k-1)^2-2n+7}{(2k-1)^2}
\end{eqnarray*}
and so $f'(k)=0$ if $k=\frac{1\pm \sqrt{2n-7}}{2}$. Since $k\geq1$, we obtain $k = \frac{1+ \sqrt{2n-7}}{2}$ as our critical point. For $k> \frac{1}{2}$, the function $f(k)$ is concave-up, so $f(\frac{1+ \sqrt{2n-7}}{2})$ is the minimum. Note that
\begin{eqnarray*}
    f\left( \frac{1 + \sqrt{2n-7}}{2} \right) &=& \frac{1 + \sqrt{2n-7}}{2} + \frac{n-3-\frac{1 + \sqrt{2n-7}}{2} }{2(\frac{1 + \sqrt{2n-7}}{2})-1} \\
     &=& \frac{1 + \sqrt{2n-7}}{2} + \frac{n-3-\frac{1 + \sqrt{2n-7}}{2} }{\sqrt{2n-7}} \\
    &=& \frac{\sqrt{2n-7}+2n-7+2n-7-\sqrt{2n-7}}{2\sqrt{2n-7}} \\
    &=& \sqrt{2n-7}.
\end{eqnarray*}
Since $\thc(H_n)$ is an integer,
\[
\min_{1\leq k<\gamma(H_n)}\left\{ k+ \left\lceil \frac{n-3-k}{2k-1} \right\rceil \right\} \geq \lceil \sqrt{2n-7} \rceil.
\]

Considering all possible values of $k$, we now have that
\[
\thc(H_n) \geq \min\left\{ \lceil \sqrt{2n-7} \rceil, \left\lceil \frac{n-8}{3} \right\rceil+3, n \right\}.
\]
First, observe that $\left\lceil \frac{n-8}{3} \right\rceil+3\leq \frac{n-8}{3} +4\leq n$ for all $n\geq 2$. So it remains to show that $\lceil \sqrt{2n-7} \rceil\leq \left\lceil \frac{n-8}{3} \right\rceil+3$ for all $n\geq 7$. To do this, we note $\left\lceil \frac{n-8}{3} \right\rceil+3=\left\lceil \frac{n+1}{3} \right\rceil$. Every $n\geq 7$ can be written as $n=3a-1$, $n=3a-2$, or $n=3a-3$ for some integer $a\geq 3$. In each of these cases, $\left\lceil \frac{n+1}{3} \right\rceil=a$.

If $n=3a-1$, $\lceil \sqrt{2n-7} \rceil=\lceil \sqrt{6a-9} \rceil$; if $n=3a-2$, $\lceil \sqrt{2n-7} \rceil=\lceil \sqrt{6a-11} \rceil$; and if $n=3a-3$, $\lceil \sqrt{2n-7} \rceil=\lceil \sqrt{6a-13} \rceil$. Since $\lceil \sqrt{6a-13} \rceil\leq \lceil \sqrt{6a-11} \rceil\leq \lceil \sqrt{6a-9} \rceil$, we only need to consider when $\lceil \sqrt{6a-9} \rceil\leq a$. Because $a$ is an integer, $\lceil \sqrt{6a-9} \rceil\leq a$ if and only if $6a-9\leq a^2$. Since this holds for $a\geq 3$, then $\lceil \sqrt{2n-7} \rceil\leq a= \left\lceil \frac{n-8}{3} \right\rceil+3$ for all $n\geq 7$.






Thus, for all $n\geq 7$,
\[
\thc(H_n) \geq \min\left\{ \lceil \sqrt{2n-7} \rceil, \left\lceil \frac{n-8}{3} \right\rceil+3, n \right\}= \lceil \sqrt{2n-7} \rceil.
\qedhere\] 
\end{proof}

Now, we consider $\dmg_k(H_n)$ and $\thd(H_n)$.

\begin{lem}\label{lem:kDMGUpperBound}
If $n$ and $k$ are integers with $n>10$ and $1 \leq k \leq n$, then \[\dmg_k(H_n) \leq \begin{cases}  \lfloor\frac{n-3}{2}\rfloor -1 & \text{if } k=1;\\ \lceil \frac{n-5-k}{2k-1} \rceil - 1 & \text{if }2 \leq k < \lceil \frac{n-4}{3}\rceil;\\ 
1 & \text{if } \lceil \frac{n-4}{3}\rceil\leq k< \lceil\frac{n-8}{3}\rceil + 2;\\
0 & \text{if }\lceil\frac{n-8}{3}\rceil + 2\leq k\leq n. \end{cases}\]
\end{lem}

\begin{proof}
Recall that by \cite{CS19}, $\dmg_1(H_n) = \lfloor\frac{n-3}{2}\rfloor -1$. Furthermore, note that $\dmg_k(H_n) = 0$ when $\gamma(H_n) = \lceil\frac{n-8}{3}\rceil + 2\leq k\leq n.$ For $2 \leq k < \lceil \frac{n-4}{3} \rceil$, we provide a cop strategy and argue that given this strategy, the robber can visit a limited number of unique vertices. In this case, place one cop on vertex $4$. Additionally, place the remaining $k-1$ cops on the tail as follows: starting at the end of the tail, divide the path into subpaths of length $2\lceil \frac{n-5-k}{2k-1} \rceil +1$ and place a cop at the center of each subpath. Thus, in each of these subpaths every vertex is distance at most $\lceil \frac{n-5-k}{2k-1} \rceil$ from a cop. Since
\[n-(k-1)\left(2\left \lceil \frac{n-5-k}{2k-1} \right \rceil +1\right)\leq 6+\left\lceil\frac{n-5-k}{2k-1}\right\rceil,\]
any remaining tail vertices not in a subpath are within $\lceil \frac{n-5-k}{2k-1} \rceil$ of the cop on vertex $4$. So if the robber places on any tail vertex, the two closest cops will move towards the robber and capture will occur in at most $\lceil \frac{n-5-k}{2k-1} \rceil$ steps. In this case, damage is at most $\lceil \frac{n-5-k}{2k-1} \rceil -1$. 

Otherwise, the robber must place on vertex $1$ since vertex $4$ dominates vertices $2,3,5$, and $6$. By remaining still, the cop on vertex $4$ limits the robber to damaging only vertex $1$. So in this case, the robber can always choose to place on the tail to maximize damage unless the tail is dominated by cops. The tail is dominated if $k-1\geq \lceil \frac{n-7}{3}\rceil$, i.e.,  $k\geq \lceil \frac{n-4}{3}\rceil$. Since we assumed $2 \leq k<\lceil \frac{n-4}{3}\rceil$, we have $\dmg_k(H_n)\leq \lceil \frac{n-5-k}{2k-1} \rceil - 1$, as desired. 

Finally, if  $\lceil \frac{n-4}{3}\rceil\leq k< \gamma(H_n)$, using the strategy above, the cops dominate the tail of $H_n$. This means that in order to damage a vertex, the robber must place on vertex $1$. Again, the cop on vertex $4$ can prevent any further damage. Thus, $\dmg_k(H_n) =1$ in this case.
\end{proof}


\begin{thm}
For each integer $a \geq 3$, if $n = 2a^2 - 2a + 6$, then $\thd(H_n) \leq \sqrt{2n - 11} - 1$.
\end{thm}

\begin{proof}
Let $k= \lceil \frac{1+ \sqrt{2n-11}}{2} \rceil$. Note that since $n=2a^2 -2a+6$,
\begin{eqnarray*}
k&=& \left\lceil \frac{1+ \sqrt{2(2a^2 -2a+6)-11}}{2} \right\rceil\\
&=& \left\lceil \frac{1+ \sqrt{(2a-1)^2}}{2} \right\rceil\\
&=& a.
\end{eqnarray*}

By Lemma \ref{lem:kDMGUpperBound}, $\thd(H_n) \leq k+\lceil \frac{n-5-k}{2k-1} \rceil -1$ for $2\leq k< \lceil \frac{n-4}{3} \rceil$. In order to verify that $2\leq \lceil \frac{1+ \sqrt{2n-11}}{2} \rceil< \lceil \frac{n-4}{3}\rceil$, observe that $a\leq \frac{2a^2-2a+6}{3}\leq \lceil \frac{n-4}{3}\rceil$ for all $a\geq 3$. Thus,

\[\thd(H_n) \leq a + \left\lceil \frac{n-5-a} {2a-1}\right\rceil - 1.\]

Since $n=2a^2-2a+6$,
\begin{eqnarray*}
\thd(H_n)&\leq & a + \left\lceil \frac{2a^2-2a+6-5-a} {2a-1}\right\rceil-1\\
&=& a + \left\lceil \frac{(2a-1)(a-1)} {2a-1}\right\rceil - 1\\
&=& 2a-2.
\end{eqnarray*}

Finally, note that $\sqrt{2n-11}= \sqrt{(2a-1)^2} = 2a-1$ and so,
\[\thd(H_n) \leq \sqrt{2n-11} -1 . \qedhere\]


\end{proof}

\begin{lem}
If $a\geq 3$ is an integer and $n = 2a^2 - 2a + 6$, then \[\left \lceil \sqrt{2n-7} \right \rceil - (\sqrt{2n-11} - 1) \geq 2.\]
\end{lem}

\begin{proof}
Note that $4a^2 -4a +5 > (2a -1)^2 = 4a^2 -4a +1 \geq 0$.
Then, we have 
\begin{align*}
\sqrt{4a^2 -4a +5} > (2a -1) &\Rightarrow \sqrt{4a^2 -4a +5} - (2a - 1) > 0\\
&\Rightarrow \left \lceil \sqrt{4a^2 -4a +5} - (2a -1) \right \rceil \geq 1\\
&\Rightarrow \left \lceil \sqrt{4a^2 -4a +5} \right \rceil -(2a-1) \geq 1\\
&\Rightarrow \left \lceil \sqrt{2n-7} \right \rceil - (\sqrt{2n-11} -1) \geq 2. \qedhere
\end{align*}
\end{proof}

\begin{cor}
For every integer $a \geq 3$, if $n = 2a^2 - 2a + 6$, then $\thd(H_n) < \thc(H_n) - 1$.
\end{cor}


\end{subsection}



Hence, we have obtained an infinite family of graphs $G$ for which $\thd(G) \leq \thc(G) -2$.  

\end{section}
\begin{section}{Further investigations of damage numbers}\label{sec:damageNumbers}

In this section, we further explore the $k$-damage number and its implications for damage throttling. First, we discuss a necessary condition for $c(G)=k+\dmg_k(G)$  and characterize graphs for which $c(G)=1+\dmg_1(G)$. Then we provide an upper bound for $\dmg_k(G)$ in terms of maximum degree, which is a generalization of a bound for $\dmg_1(G)$ in $\cite{CS19}$.

In the proof of Proposition \ref{prop:c(G)leqthd(G)}, we showed $c(G)\leq k+\dmg_k(G)$ for each $1\leq k\leq n$. Proposition \ref{prop:copGammaThrot} considers the case where $c(G)=\thd(G)$ and shows that for such graphs, either $c(G)=\gamma(G)$ or $\thd(G)<\thc(G)-1$. It is easy to observe that $c(G)=\thd(G)$ if and only if there exists an integer $\ell$ such that $c(G)=\ell+\dmg_{\ell}(G)$. So characterizing instances where $c(G)= k+\dmg_k(G)$ for some $k$ is useful, as it produces a class of graphs which are good candidates for achieving a gap of two or greater between $\thc(G)$ and $\thd(G)$. First, we prove a helpful lemma.

\begin{lem}\label{c(G)leq k+dmg_k(G)-1}
Suppose $G$ is a connected graph. For each integer $1 \leq k \leq n$, if $\dmg_k(G) \geq 2$, then $c(G) \leq k +\dmg_k(G) -1$.
\end{lem}
\bpf
By Proposition \ref{prop:c(G)leqthd(G)}, we know that $c(G) \leq \thd(G)$. Apply the same cop strategy used in Proposition \ref{prop:c(G)leqthd(G)}, but with $\dmg_k(G)-1$ undercover cops.  As $G$ is a connected graph, the $\dmg_k(G)-1$ undercover cops can catch the robber because once these cops are deployed, every vertex in the damaged area will either be occupied by a cop, or adjacent to only vertices containing cops. 
\epf


Applying Lemma \ref{c(G)leq k+dmg_k(G)-1}, we see the following necessary condition for $c(G)=k+\dmg_k(G)$.

\begin{obs}
For any connected graph $G$ and integer $1\leq k\leq n$, if $c(G)=k+\dmg_k(G)$, then $\dmg_{k}(G)\leq 1$. Furthermore, either $\dmg_{c(G)}(G)=0$ or $\dmg_{c(G)-1}(G)=1$.
\end{obs}

For the case where $k=1$, we now prove a complete characterization of connected graphs $G$ which achieve $c(G)=1+\dmg_1(G)$.

\begin{prop}\label{prop:c(G)=dmg(G) +1}
If $G$ is a connected graph, then $c(G) = 1+ \dmg_1(G)$ if and only if G has a dominating vertex, or $c(G)=2$ and there exists $v \in V(G)$ such that for every $w \notin N[v]$, there exists a vertex $x \in N[v]$ such that $N(w) \subseteq N[x]$.
\end{prop}
\bpf
If $c(G)=\dmg_1(G)+1$, then by Lemma \ref{c(G)leq k+dmg_k(G)-1}, $\dmg_1(G) \leq 1$, and therefore $c(G) \leq 2$.  If $c(G)=1$, then $\dmg_1(G)=0$ and $G$ must have a dominating vertex.  Otherwise, if $c(G)=2$, then $\dmg_1(G)=1$.  If the robber is ever able to move without being captured in the same round, then $\dmg_1(G)>1$. Therefore, there must exist an initial placement of one cop which restricts the robber to a single ``safe" vertex. In other words, there exists some initial cop placement $\{v\} \subseteq V(G)$ such that no matter where the robber places, the cop can always move or stay still in round 1 to prevent the robber from moving. To guarantee this, there must exist $v \in V(G)$ such that for every $w \notin N[v]$, there exists a vertex $x \in N[v]$ such that $N(w) \subseteq N[x]$.

Conversely, if $G$ has a dominating vertex, then $\dmg_1(G)=0$ and $c(G)=1$, which means $c(G)=\dmg_1(G)+1$. If $c(G)=2$ and there exist $v \in V(G)$ such that for every $w \notin N[v]$, there exists a vertex $x \in N[v]$ such that $N(w) \subseteq N[x]$, then the cops can prevent damage to all but one vertex. So $\dmg_1(G)=1$ and therefore, $c(G)=\dmg_1(G)+1$.
\epf

It is worth noting that the set of graphs $G$ for which $c(G)=2$ and $\dmg_1(G) = 1$ is not empty. In fact, there is an infinite family of graphs that satisfy these conditions. For example, we can add an edge between a vertex $v \in V(C_4)$ and every vertex in an arbitrary graph $H$. Furthermore, we can add any number of leaves to the two neighbors of $v$ in the cycle. This family is illustrated in Figure \ref{fig:4cycleplus}.

\begin{figure}[H] \begin{center}
\scalebox{.6}{\includegraphics{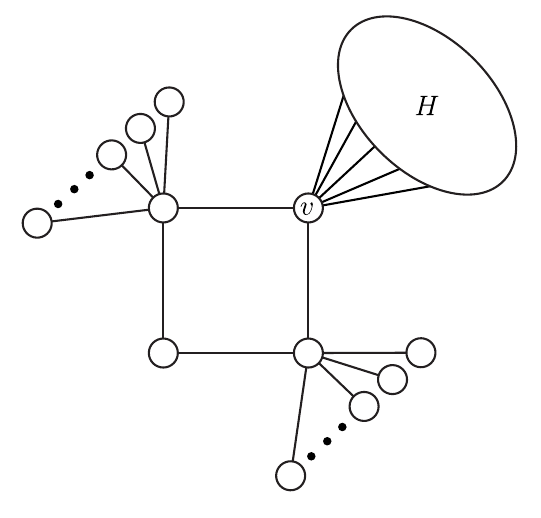}}\\
\caption{An infinite family of graphs satisfying the conditions in Proposition \ref{prop:c(G)=dmg(G) +1}.}\label{fig:4cycleplus} 
\end{center}
\end{figure}
\vspace{-.2 in}
In $\cite{CS19}$, the authors prove the following upper bound for $\dmg_1(G)$. 

\begin{prop}\cite{CS19}
For a graph $G$ on $n$ vertices, $\dmg_1(G)\leq n-\Delta(G) - 1$.
\end{prop}

We now prove an analogous upper bound for $\dmg_k(G)$ and apply it to provide an upper bound for $\thd(G)$. For a graph $G$, let $S_k$ be the set of $k$-vertex subsets $S$ of $V(G)$.

\begin{prop}\label{GenUppBndDmgk}
For all graphs $G$ on $n$ vertices, \[\dmg_k(G)\leq \min_{S\in S_k} \{n-|N[S]|\}.\]
\end{prop}

\begin{proof}
Place $k$ cops on the $k$ vertices of $S\in S_k$. By remaining still, a cop placed on $v\in S$ protects the vertices in $N[v]$ from being damaged. Therefore, using this cop placement, $|N[S]|$ vertices remain undamaged.
\end{proof}

\begin{cor}\label{GenUppBndThd}
For all graphs $G$ on $n$ vertices, \[\thd(G)\leq \min_{1\leq k\leq n}\left\{k+ \min_{S\in S_k} \{n-|N[S]|\}\right\}.\]
\end{cor}


\end{section}

\begin{section}{Concluding Remarks}\label{sec:conclusion}

As shown in Section \ref{sec:diff}, there are infinite families of graphs $G$ such that $\thd(G)\leq \thc(G)-2$. However, we were not able to verify the existence of an infinite family of graphs $G$ satisfying $\thd(G)=\thc(G)-a$ for any $a\geq 3$ (we do provide a single graph for which $\thd(G)=\thc(G)-3$ in Theorem \ref{thm:GapOf3Graph}). Despite this, we believe such families exist and it would be interesting to find them.

In \cite{CS19}, the authors ask the question: which graphs $G$ satisfy 
$\dmg_1(G)=n-\Delta(G)-1$? We observe that graphs for which $\Delta(G)=n-1$ or $\Delta(G)=n-2$ achieve this equality, but this does not characterize all such graphs. For example, the graphs in Figure \ref{fig:jellyfish} achieve equality in this bound, but have $\Delta(G)=n-3$. 
\begin{figure}[H] \begin{center}
\scalebox{.7}{\includegraphics{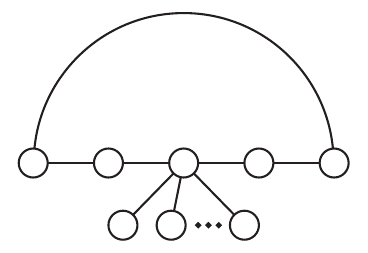}}\\
\caption{An infinite family of graphs $G$ with $\dmg_1(G) = |V(G)| - \Delta(G) - 1$, $\rad(G) = 2$, and $\Delta(G) = |V(G)| - 3$.}\label{fig:jellyfish} 
\end{center}
\end{figure}
\vspace{-.2 in}
We now propose similar questions for the generalized bounds in Proposition \ref{GenUppBndDmgk} and Corollary \ref{GenUppBndThd}.
\begin{quest}\label{quest:GenBounds}
Which graphs $G$ satisfy $\dmg_k(G)=\underset{S\in S_k}{\min} \{n-|N[S]|\}$ for some integer $k$ and furthermore, which graphs satisfy $\thd(G)= \underset{1\leq k\leq n}{\min}\left\{k+ \underset{S\in S_k}{\min} \{n-|N[S]|\}\right\}$?
\end{quest} 
Observe that in Question \ref{quest:GenBounds}, a graph $G$ that realizes the first equality does not necessarily realize the second. However, a graph that satisfies the second equality must satisfy the first for some integer $k$.



\end{section}

\end{document}